\documentclass[10pt,reqno]{amsart}
\usepackage{amssymb,amsmath,amsthm,graphicx,verbatim,amsbsy,mathtools}
\DeclarePairedDelimiter{\floor}{\lfloor}{\rfloor}
\DeclarePairedDelimiter{\ceil}{\lceil}{\rceil}

\newcommand{\J}{{\bf J}}
\newcommand{\Z}{{\bf Z}}
\newcommand{\R}{{\bf R}}

\newcommand{\cl}{{\rm cl}_{L^2}\,}

\theoremstyle{plain}
\newtheorem{theorem}{Theorem}
\newtheorem{lemma}[theorem]{Lemma}

\theoremstyle{definition}

\theoremstyle{remark}

\def\hypergeom#1#2#3#4#5{{}_#1 F_{#2}\left({#3\atop#4}; \ #5\right)}

\title{On Alpert Multiwavelets}

\author[J.~Geronimo]{Jeffrey~S.~Geronimo}
\thanks{JSG was supported in part by  Simons Foundation Grant.}
\address{JSG, School of Mathematics, Georgia Institute of Technology,
Atlanta, GA 30332--0160, USA}
\author[F.~Marcellan]{Francisco~Marcell\'an}
\address{FM, Departmento de Mathem\'aticas, Universidad Carlos III de Madrid,
28911, Legan\'es, Spain}
\thanks{FM was supported by grant MTM2012-36732-C03-01 from the Direcc\'ion General de Investigaci\'on, Minsiterio de Econom\'ia y Competitividad of Spain. }
\begin{document} 

\maketitle

\begin{abstract} The multiresolution analysis of Alpert is
considered. Explicit formulas for the entries in the matrix coefficients of 
the refinement equation are given in terms of hypergeometric
functions. These entries are shown to solve  generalized eigenvalue
equations as well as partial difference equations. The matrix coefficients
in the wavelet equation are also considered and conditions are given to
obtain a unique solution. 
 \smallskip 



\end{abstract}

\section{Introduction}
The theory of wavelets has had a broad and lasting impact on various areas of
mathematics and engineering such as numerical analysis, signal processing,
and harmonic analysis \cite{C},\cite{CW},\cite{mal},\cite{mey}. The most
well known wavelet may
be the Haar wavelet which is not continuous and one of the great achievements
in the area is Daubechies' construction of compactly supported, orthogonal 
wavelets that are at least continuous \cite{da}. The theory of one variable 
multiwavelets \cite{ala}, \cite{dgh}, \cite{dghm}, \cite{hkm}, \cite{he}  is an extension of wavelet theory 
to the case of when there are several scaling functions instead of just one. 
This extra flexibility allows the construction of
piecewise polynomial scaling functions and wavelets that are compactly
support orthogonal and at least continuous \cite{dgha}. 
The scaling function associated
with the Haar wavelet is the constant function supported on $[0,1]$ and
zero elsewhere and the linear space associated with this function is the space
of piecewise constant polynomials with integer knots. The extension of this
space to higher degree polynomials gives the space of piecewise
polynomials of degree n with integer knots and an orthogonal basis for this
space are the Legendre polynomials restricted to $[0,1]$  and their integer translates. Alpert first 
developed the multiresolution analyses associated with these spaces and 
applied them to various problem in integral equations \cite{alb} and
numerical analysis \cite{dd} and \cite{st} An important equation in 
multiresolution analysis is the refinement equation which links the scaling 
functions on one level to their scaled versions. Here we examine in more 
detail the coefficients in the refinement equation associated with the Alpert
multiresolution analysis with the intent of obtaining formulas for these
coefficients as well as recurrence relations. These lead to combinatorial
identities and orthogonality relations that seem to have been unnoticed.
In section 1 we review Alpert's   multiresolution analyses  and
make contact with the Legendre polynomials. In section 2 we derived
various representations for the entries of the matrices in the refinement
equation and discuss the orthogonality relations satisfied by these
coefficients. In section 3 we develop recurrence formulas satisfied by
these coefficients and show that they give rise to some generalized
eigenvalue problems. In section 4 we investigate the Fourier transform 
of the scaling functions which turns out to be related to Bessel
functions of half integer order. Using some identities satisfied by 
Bessel functions we arrive at other recurrences satisfied by the entries 
in the refinement matrices. Also in this section are the identities that arise 
due to polynomial reproduction. Finally in section 5 we consider the matrices 
in the wavelet equation associated with these multiresolution analyses. These
matrices must satisfy certain conditions which follow from the orthogonality
of the wavelets to the scaling functions and to the other wavelets. We present 
natural conditions in order for there to be a unique solution to these 
equations.

\section{Preliminaries}
Let
$\phi^0,\dots,\phi^r$ be compactly supported $L^2$-functions, and suppose that
$V_0 = \cl{\rm span}\{\phi^i({\cdot}-j): i = 0,1,\dots,r,\ j\in\Z\}$.  Then
$V_0$ is called a {\it finitely generated shift invariant\/ {\rm (FSI)}
space}.  Let $(V_p)_{p\in\Z}$ be given by
$V_p = \{\phi(2^p{\cdot}): \phi\in V_0\}$.  Each space $V_p$ may be
thought of as approximating $L^2$ at a different resolution depending
on the value of $p$.  The sequence $(V_p)$ is called a
{\it multiresolution analysis} \cite{da,ghm,gl} generated by {$\phi^0,\dots,\phi^r$} if
(a) the spaces are nested, $\cdots\subset V_{-1}\subset V_0\subset V_1
\subset\cdots$, and (b) the generators $\phi^0,\dots,\phi^r$ and 
their integer translates form a Riesz basis for $V_0$.  Because of (a)
and (b) above, we can write
\begin{equation}\label{nest}
V_{j+1} = V_j \oplus W_j \quad\forall j\in\Z. 
\end{equation}
The space $W_0$ is called the {\it wavelet space}, and if
$\psi^0,\dots,\psi^r$ generate a shift-invariant basis for $W_0$, then
these functions are called {\it wavelet functions}.
If, in addition, $\phi^0,\dots,\phi^r$ 
and their integer translates form an orthogonal basis for $V_0$,
then $(V_p)$ is called an {\it orthogonal MRA}.  Let $S_{-1}^n$ be the space of
polynomial splines of degree $n$ continuous except perhaps at the integers,
and set $V_0^n = S_{-1}^n \cap L^2(\R)$.  With $V_p^n$ as above these spaces form 
a multiresolution analysis. If $n=0$ the  multiresolution analysis obtained is 
associated with the Haar wavelet while for $n>0$ they were  introduced by 
Alpert \cite{ala, alb}. 
If we let 
$$
\phi_j(t)=\begin{cases} \hat p_j(2t-1),&  0\le t<1\\ 0,& \text{elsewhere},\end{cases} 
$$
where $\hat p_j(t)$ is the  Legendre polynomial \cite{sz} of degree $j$ orthonormal on $[-1,1]$  with 
positive leading coefficient i.e. $\hat p_j(t)=k_j t^j +\text{lower degree terms}$ with $k_j>0$ and
$$
\int_{-1}^1 \hat p_j(t)\hat p_k(t)dt=\delta_{k,j},
$$ 
then
\begin{equation}\label{Phi}
\Phi_n=\left[\begin{matrix}\phi_0&\cdots&\phi_n\end{matrix}\right]^T,
\end{equation}
 and its integer translates form an orthogonal basis for $V_0$. For the convenience in later computations we set 
\begin{equation}\label{PN}
P_n(t)=\left[\begin{matrix}\hat p_0(t)\\\vdots\\\hat p_n(t)\end{matrix}\right]\chi_{[0,1]}.
\end{equation}
Equation~\eqref{nest} implies the existence of the {\it refinement} equation,  
\begin{equation}\label{refin}
\Phi_n(\frac{t}{2})=C^n_{-1}\Phi_n(t)+C^n_1\Phi_n(t-1),
\end{equation}
where the $C_{i}^n,\ i=-1,1$ are $(n+1)\times(n+1)$ matrices. The orthonormality of the entries in $\Phi_n(\frac{t}{2})$ implies that
\begin{equation}\label{mc1ortho}
2I_{n+1}=C^n_{-1}{C^n_{-1}}^T+C^n_1{C^n_1}^T,
\end{equation}
where $I_n$ is the $n\times n$ identity matrix
and $A^T$ is the transpose of $A$.
In terms of the entries of $P_n$ we see, 
\begin{equation}\label{refl2}
\hat p_i(t)=\sum_{j=0}^i(C^n_{-1})_{i,j}\hat p_j(2t+1)|_{[-1,0)}+\sum_{j=0}^i(C^n_1)_{i,j}\hat p_j(2t-1)|_{[0,1]},
\end{equation}
for $-1\le t\le1$.
In order to exploit the symmetry of the Legendre polynomials we shift $t\to t+1$
so that 
\begin{align}\label{refsym}
\Phi_n(\frac{t+1}{2})=P_n(t)&=C^n_{-1}\Phi_n(t+1)+C^n_1\Phi_n(t)\nonumber\\&=C^n_{-1}P_n(2t+1)+C^n_1P_n(2t-1).
\end{align}
These polynomials have the following representation in terms of
a ${}_2 F_1$ hypergeometric function \cite{sz}(p. 80),
\begin{equation}\label{hyperone}
p_n(t) = \frac{2^nn!}{(n+1)_n}
  \,\hypergeom21{-n,\ n+1}{1}{\frac{1-t}{2}},
\end{equation}
where formally,
$$
\hypergeom{p}{q}{a_1,\ \dots\ a_p}{b_1,\ \dots\ b_q}{t}
  = \sum_{i=0}^{\infty}\frac{(a_1)_i\dots(a_p)_i}{(b_1)_i\dots(b_q)_i(1)_i}t^i
$$ 
with $(a)_0=1$ and $(a)_i = (a)(a+1)\ldots(a+i-1)$ for $i>0$.  Since one of the numerator parameters in the definition of $p_n$ is a negative integer 
the series in  equation~\eqref{hyperone} has only finitely many terms. The
relation between $\hat p_n$ and $p_n$ is given by,
\begin{equation}\label{orthonormal}
\hat p_n(t)=\frac{\sqrt{2n+1} (2n-1)!!}{\sqrt{2}n!} p_n(t).
\end{equation}

A representation that makes the symmetry of the Legendre polynomials
manifest is \cite{sz}(p. 83)
\begin{equation}\label{lsym}
p_{2n}(x)=(-1)^n\frac{(1/2)_n}{(n+1/2)_n}\hypergeom21{-n,\ n+1/2}{1/2}{x^2},
\end{equation}
and
\begin{equation}\label{lasym}
p_{2n+1}(x)=(-1)^n\frac{(3/2)_n x }{(n+3/2)_n}\hypergeom21{-n,\ n+3/2}{3/2}{x^2}.
\end{equation}
Finally we recall the well known recurrence formula satisfied by the monic Legendre polynomial,
\begin{equation}\label{lrec}
p_{n+1}(t)=tp_n(t)-\frac{n^2}{(2n+1)(2n-1)}p_{n-1}(t).
\end{equation}
\subsection{Coefficient Representations}
Since the Legendre polynomial are symmetric or antisymmetric we need need only compute $C_1$ which equation~\eqref{refsym} shows is given by
\begin{equation}\label{c1}
C^n_1=\int_{0}^1P_n(t)P_n(2t-1)^T dt,
\end{equation}
so that
\begin{equation}\label{cij}
(C^n_1)_{i,j}=\int_{0}^1 \hat p_i(t)\hat p_j(2t-1) dt.
\end{equation}
In the above equation we index the entries in $C^n_1$ beginning with $i=0,j=0$. Because of the orthogonality of the Legendre polynomials to powers of $t$ less that their degree the above integral is equal to zero for $i<j$. Summarizing we find
\begin{lemma}\label{cosym} Let $C_1^n$ and $C_{-1}^n$ be the matrix coefficients in the 
above refinement equation. Then $C^n_1$ is a lower triangular matrix with positive diagonal entries. Furthermore 
\begin{equation}\label{symmco}
(C_{-1}^n)_{i,j}=(-1)^{i+j}(C_1^n)_{i,j},\quad i,\ j\ge0,
\end{equation} 
which gives the orthogonality relations
\begin{equation}\label{o1}
0=((-1)^{i+k}+1)\sum_{j=0}^i(C^n_1)_{i,j}(C^n_1)_{k,j} \ k> i,
\end{equation}
and
\begin{equation}\label{or2}
1=\sum_{j=0}^i(C^n_1)_{i,j}(C^n_1)_{i,j}.
\end{equation}
\end{lemma}
We examine the above integral using monic polynomials $p_i$ which in terms of hypergeometric functions is
\begin{align}\label{I1}
I^1_{i,j}&=\int_{0}^1 p_i(t) p_j(2t-1)dt\\&=
\frac{2^{i+j}(1)_i(1)_j}{(i+1)_i(j+1)_j}I^2_{i,j},
\end{align}
where
\begin{equation*}
 I^2_{i,j}=(-1)^j\int_0^1\hypergeom21{-i,\ i+1}{1}{\frac{1-t}{2}}\hypergeom21{-j,\ j+1}{1}{t}dt.
\end{equation*}
The symmetry of the Legendre polynomials has been used to obtain the last expression. From the definition of the hypergeometric functions we find after integration,
\begin{equation*}
 I^2_{i,j}=
(-1)^j\sum_{k=0}^i\sum_{n=0}^j\frac{(-i)_k(i+1)_k}{(1)_k(1)_k 2^k}\frac{(-j)_n(j+1)_n}{(1)_n(1)_n} \frac{k! n!}{(n+k+1)!}.
\end{equation*}
Since $(n+k+1)!=(k+2)_n(1)_{k+1}$ the sum on $n$ equals
$\hypergeom21{-j,\ j+1}{k+2}{1}=\frac{(k-j+1)_j}{(k+2)_j}$ by the
Chu-Vandermonde formula \cite{ba}(p. 3) so,
$$
I^2_{i,j}=(-1)^j\sum_{k=j}^i\frac{(-i)_k(i+1)_k}{(1)_k(1)_{k+1} 2^k}\frac{(k-j+1)_j}{(k+2)_j}.
$$
where the fact that $(k-j+1)_j=0$ for $k<j$ has been used to obtain the
equality. Shifting $k$ by $k-j$, then using the identities
$(a+j)_k=(a)_k(a+j)_k$ with $a= -i, i+1$,
$
(k+1)_j=\frac{(1)_j(j+1)_k}{(1)_k}, \ \text{and} \ (k+j+2)_j=\frac{(1)_{2j+1}(2j+2)_k}{(1)_{j+1}(j+2)_k},
$
yields
\begin{align*}
I^2_{i,j}&=(-1)^j\frac{(-i)_j(i+1)_j}{(1)_{2j+1}2^j}\sum_{k=0}^{-i+j}\frac{(-i+j)_k(i+j+1)_k}{(1)_k(2j+2)_k 2^k}\\&=(-1)^j\frac{(-i)_j(i+1)_j}{(1)_{2j+1}2^j}\hypergeom21{-i+j,\ i+j+1}{2j+2}{\frac{1}{2}}.
\end{align*}
Substituting this into equation~\eqref{I1} yields
\begin{equation}\label{I1hyp}
I^1_{i,j}=\frac{2^i(1)_i(i+1)_j(1)_j}{(1)_{i-j}(i+1)_i(j+1)_j(1)_{2j+1}}\hypergeom21{-i+j,\ i+j+1}{2j+2}{\frac{1}{2}},
\end{equation}
where we have used the identity $(-1)^j(-i)_j=(1)_i/(1)_{i-j}$. This shows that
\begin{align}\label{c1ij}
(C^n_1)_{i,j}&=\frac{(2i-1)!!(2j-1)!!\sqrt{(2i+1)(2j+1)}}{(1)_j(1)_i}I^1_{i,j}\nonumber\\&=l_{i,j}\hypergeom21{-i+j,\ i+j+1}{2j+2}{\frac{1}{2}},
\end{align}
where
\begin{equation}\label{prefact}
l_{i,j}=\sqrt{\frac{2i+1}{2j+1}}\frac{(i+j)!}{2^j(2j)!(i-j)!}.
\end{equation}

When the parity of $i$ and $k$ are the same the sum in equation~\eqref{o1}
must be equal to zero and it is  easy to check that the sum in
\eqref{o1} is not in general equal to zero when $i$ and $k$ are of
different parities. If we set
$n=i-j$ in the hypergeometric function above the function becomes
\begin{equation}\label{jac}
2^n\frac{(2j+2)_n}{(n+2j+1)_n}\hypergeom21{-n,\ n+2j+1}{2j+2}{\frac{1}{2}}=p^{(2j+1,-1)}_n(0),
\end{equation}
where $p^{(\alpha,\beta)}_n(x)$ is the monic Jacobi polynomial. Since $\beta =-1,\ p^{(2j+1,-1)}_n(x)$ is not in the standard class of Jacobi orthogonal polynomials, 
furthermore in the discrete orthogonality above both the degree and the order
are changing. The representation given in equation~\eqref{c1ij} suggests an easy recurrence formula in $i$ but not so simple in $j$. A useful representation for the above  hypergeometric function that simplifies the dependence on $j$ maybe obtained by using the transformation $\hypergeom21{-n,\ b}{c}{x}=\frac{(b)_n}{(c)_n}(-x)^n \hypergeom21{-n,\ -c-n+1}{-b-n+1}{\frac{1}{x}}$ which yields,
\begin{align}\label{rep2}
\hat l_{i,j} \hypergeom21{-(i-j),\ i+j+1}{2j+2}{\frac{1}{2}}&=\hypergeom21{-(i-j),\ -i-j-1 }{-2i}{2}\nonumber\\&=\hypergeom21{-n,\ -2i+n-1}{-2i}{2},
\end{align} 
where $\hat l_{i,j}=(-2)^{i-j}\frac{(i+j+1)!(i+j)!}{(2j+2)!(2i)!}$ and $n=i-j$. The last equality shows that the hypergeometric function is 
related to Krawtchouk polynomials \cite{aar}(p. 347).

The orthogonality relation \eqref{o1} is nontrivial only among the even and
odd rows of $C_1^n$. To take this into account we use the expressions~\eqref{lsym} and \eqref{lasym}. Furthermore in order to make
apparent the polynomial character in $j$ of the resulting hypergeometric
function we use the transformation leading to equation~\eqref{rep2}. 
In this case
\begin{align}\label{ieven}
 I^1_{2i,j}&=(-1)^{i+j}2^j\frac{(1)_i(1)_j}{(-i)_i(j+1)_j}\int_0^1
 t^{2i}\hypergeom21{-i,\ -i+1/2}{-2i+1/2}{1/t^2}\hypergeom21{-j,\
   j+1}{1}{t}\nonumber\\&=(-1)^{i+j}2^j\frac{(1)_i(1)_j}{(-i)_i(j+1)_j}\sum_{k=0}^i\frac{(-i)_k(-i+1/2)_k}{(1)_k(-2i+1/2)_k}S^e_{j,n},
\end{align}
where
\begin{align*}
S^e_{j,j}&=\sum_{n=0}^j\frac{(-j)_n (j+1)_n}{(1)_n(1)_n}\int_0^1 t^{2(i-k)+n}dt\\&=\frac{1}{2(i-k)+1}\sum_{n=0}^j\frac{(-j)_n (j+1)_n(2(i-k)+1)_n}{(1)_n(1)_n(2(i-k)+2)_n}.
\end{align*}
The last sum is $\hypergeom32{-j,\ j+1,\ 2(i-k)+1}{1,\
  2(i-k)+2}{1}=\frac{(-j)_j (2(i-k)-j+1)_j}{(1)_j(2(i-k)+2)_j}$ where the
Pfaff-Saalschutz formula \cite{ba}(p.9) has been used since the hypergeometric function is balanced (i.e the sum of the numerator parameter is one less than the sum of the denominator parameters). Substitution of the above result in equation~\eqref{ieven} yields
\begin{equation}\label{eveni}
I^1_{2i,j}=(-1)^{i+j}2^j\frac{(1)_i(-j)_j}{(-i)_i(j+1)_j}\hat S_{2i,j},
\end{equation}
where
\begin{equation}\label{iieven}
\hat S_{2i,j}= \sum_{k=0}^i\frac{(-i)_k(-i+1/2)_k (2(i-k)-j+1)_j}{(1)_k(-2i+1/2)_k (2(i-k)+1)_{j+1}}.
\end{equation}
Now it is most convenient to consider $j$ even or odd. For $j\to 2j$ the above sum is equal to zero for $i-j<k$. Thus
$$
\hat S_{2i,2j}=\sum_{k=0}^{i-j}\frac{(-i)_k (-i+1/2)_k(2(i-k)-2j+1)_{2j}}{(1)_k(-2i+1/2)_k(2(i-k)+1)_{2j+1}}.
$$
For $m=0,\ 1$ we have the equations,
\begin{align*}
(2(i-k-j)+1+m)_{2j-m}&=2^{2j-m}\frac{(i-k-j+\frac{m+1}{2})(i-k-j+1+\frac{m}{2})\cdots(i-\frac{1}{2})(i)}{(-i+\frac{1}{2})_k(-i)_k}\\&=(-1)^m2^{2j-m}\frac{((-i+j+\frac{-m+1}{2})_k(-i+j-\frac{m}{2})_k (-i)_j(-i+\frac{1}{2})_{j-m}}{(-i+1/2)_k(-i)_k},
\end{align*}
and
$
(2(i-k)+1)_{2j+1-m}=2^{2j+1-m}\frac{(-i+1/2)_k(i+1/2)_{j+1-m}(-i)_k(i+1)_j}{(-i-j+m-1/2)_k(-i-j)_k}.
$
Thus with $m=0$,
\begin{align*}
\hat S_{2i,2j}&=\frac{1}{2}\frac{(-i)_j(-i+1/2)_j}{(i+1/2)_{j+1}(i+1)_j}\sum_{k=0}^{i-j}\frac{(-i+j)_k(-i+j+1/2)_k(-i-j-1/2)_k(-i-j)_k}{(1)_k(-2i+1/2)_k(-i)_k(-i+1/2)_k}\nonumber\\&=\frac{1}{2}\frac{(-i)_j(-i+1/2)_j}{(i+1/2)_{j+1}(i+1)_j}\hypergeom43{-i+j,\ -i+j+1/2,\ -i-j-1/2,\ -i-j}{-2i+1/2,\ -i,\ -i+1/2}{1}.
\end{align*}
Substitution of this into equation~\eqref{iieven} yields
\begin{align}\label{I2i2j}
 I^1_{2i,2j}&=2^{2j-1}\frac{(-i)_j(-i+1/2)_j(2j)!}{(i+1/2)_{j+1}(i+1)_j(2j+1)_{2j}}\nonumber\\&\hypergeom43{-(i-j),\
   -i+j+1/2,\ -i-j-1/2,\ -i-j}{-2i+1/2,\ -i,\ -i+1/2}{1}.
\end{align}
With $j\to2j-1$ in equation~\eqref{eveni} and $m=1$ in the above identities we obtain,
\begin{align}\label{I2i2jm1}
 I^1_{2i,2j-1}&=2^{2j-2}\frac{(-i)_j(-i+1/2)_{j-1}(2j-1)!}{(i+1/2)_j(2j)_{2j-1}(i+1)_j}\nonumber\\&\hypergeom43{-(i-j),\ -i+j-1/2,\ -i-j+1/2,\ -i-j}{-2i+1/2,\ -i,\ -i+1/2}{1}.
\end{align}
Similar manipulations for $i$ odd lead to,
\begin{align}\label{I2ip12j}
 I^1_{2i+1,2j}&=2^{2j-1}\frac{(-i)_j(-i-1/2)_j(2j)!}{(i+3/2)_j(i+1)_{j+1}(2j+1)_{2j}}\nonumber\\&\hypergeom43{-(i-j),\ -i+j-1/2,\ -i-j-1,\ -i-j-1/2}{-2i-1/2,\ -i,\ -i-1/2}{1},
\end{align}
and,
\begin{align}\label{I2ip12jp1}
 I^1_{2i+1,2j+1}&=2^{2j}\frac{(-i)_j(-i-1/2)_{j+1}(2j+1)!}{(i+3/2)_{j+1}(i+1)_{j+1}(2j+2)_{2j+1}}\nonumber\\&\hypergeom43{-(i-j),\ -i+j+1/2,\ -i-j-1,\ -i-j-3/2}{-2i-1/2,\ -i,\ -i-1/2}{1}.
\end{align}
Collecting the above computations gives,
\begin{theorem}\label{repres}
The entries in the matrix $C^n_1$ have the following representations,
\begin{align}\label{rep1}
(C^n_1)_{i,j}&=\frac{\sqrt{(2i+1)(2j+1)}(i+j)!}{2^j(2j+1)!(i-j)!}\hypergeom21{-i+j,\ i+j+1}{2j+2}{\frac{1}{2}}\nonumber\\&=(-1)^{i-j}\frac{\sqrt{(2i+1)(2j+1)}(2i)!}{2^i(i+j+1)!(i-j)!}\hypergeom21{-i+j,\ -i-j-1}{-2i}{2}.
\end{align} 
Alternatively
\begin{align}\label{rep2e}
&(C^n_1)_{2i,j}=W_{2i,j}\nonumber\\&\hypergeom43{-i+\ceil{\frac{j}{2}},\ -i+\floor{\frac{j}{2}}+1/2,\ -i-\ceil{\frac{j}{2}}-1,\ -i-\floor{\frac{j}{2}}-1/2}{-2i-1/2,\ -i,\ -i-1/2}{1}
\end{align}
and
\begin{align}\label{rep2o}
&(C^n_1)_{2i+1,j}=W_{2i+1,j}\nonumber\\&\hypergeom43{-i+\floor{\frac{j}{2}},\ -i+\ceil{\frac{j}{2}}-1/2,\ -i-\floor{\frac{j}{2}}-1,\ -i-\ceil{\frac{j}{2}}-1/2}{-2i-1/2,\ -i,\ -i-1/2}{1},
\end{align}
with
$$
W_{2i,j}=K_{2i,j}\frac{2^{j-1}j!(-i)_{\ceil{\frac{j}{2}}}(-i-\frac{1}{2})_{\ceil{\floor{j}{2}}}}{(i+\frac{1}{2})_{\floor{\frac{j}{2}}+1}(i+1)_{\ceil{\frac{j}{2}}}(j+1)_{j}},
$$
$$
W_{2i+1,j}=K_{2i+1,j}\frac{2^{j-1}j!(-i)_{\floor{\frac{j}{2}}}(-i-\frac{1}{2})_{\ceil{\frac{j}{2}}}}{(i+\frac{3}{2})_{\ceil{\frac{j}{2}}}(i+1)_{\floor{\frac{j}{2}}+1}(j+1)_{j}},
$$
and
$$
K_{i,j}=\frac{(2i-1)!!(2j-1)!!\sqrt{(2i+1)(2j+1)}}{(1)_i(1)_j}.
$$
In all cases the above hypergeometric functions are balanced. 
Also the above functions satisfy the orthogonality relations given by equations~\eqref{o1} and \eqref{or2}.

\end{theorem}

The values of $(C^n_1)_{i,j}$ for $j=i,i-1,$ and $i-2$ with $n>2$ are simple and given by,
\begin{equation}\label{cnm1}
(C^n_1)_{i,i}=\frac{1}{2^i},\quad (C^n_1)_{i,i-1}=\frac{\sqrt{(2i+1)(2i-1)}}{2^i}, 
\end{equation}
and
\begin{equation}\label{cnm2}
(C^n_1)_{i,i-2}=\frac{(i-2)\sqrt{(2i+1)(2i-1)}}{2^i}.
\end{equation}
For $n>1$ we find using Kummer's theorem \cite{ba}(p. 9),
\begin{align*}
(C^n_1)_{i,0}&=\sqrt{2i+1}\frac{\Gamma(3/2)}{\Gamma((2-i)/2)\Gamma((i+3)/2)}\nonumber\\&=\bigg\{\begin{matrix}0,&\qquad i\ \text{even},\ i>0,\\ (-1)^{\frac{i-1}{2}}\frac{\sqrt{2i+1}}{2}\ (\frac{1}{2})_{\frac{i-1}{2}}/((i+1)/2)!,&\qquad i\ \text{odd},\ i>0,\end{matrix}
\end{align*}
where $\Gamma$ is the Gamma function. That $(C_1^n)_{2i,0}=0$ also follows from the symmetry and orthogonality of the Legendre polynomials. 
For the simplest case when $n=0$
i.e. piecewise constant scaling functions we find that
$$
C^0_1=1.
$$
For other $n$ we find,
$$
C^1_1=\left(\begin{matrix} 1&0\\  \frac{\sqrt{3}}{2}&\frac{1}{2}\end{matrix}\right),\ 
C^2_1=\left(\begin{matrix} 1&0&0\\ \frac{\sqrt{3}}{2}&\frac{1}{2}&0\\0&\frac{\sqrt{15}}{4}&\frac{1}{4}\end{matrix}\right),\
\text{and} \
C^3_1=\left(\begin{matrix} 1&0&0&0\\ \frac{\sqrt{3}}{2}&\frac{1}{2}&0&0\\0&\frac{\sqrt{15}}{4}&\frac{1}{4}&0\\-\frac{\sqrt{7}}{8}&\frac{\sqrt{21}}{8}&\frac{\sqrt{35}}{8}&\frac{1}{8} \end{matrix}\right).
$$

\section{Recurrence formulas and Generalized Eigenvalue problem}

The contiguous relations for hypergeometric functions give recurrence
formulas among the entries in the matrix $C^n_1$ which we now study.
A useful and well known relation (\cite{aar} equation~(2.5.15)) that ${}_2 F_1$ hypergeometric functions
satisfy is the following,
\begin{align}\label{cont}
e_1\hypergeom21{a-1,\ b+1}{c}{x}&=e_2\hypergeom21{a,\ b}{c}{x}\nonumber\\&+e_3\hypergeom21{a+1,\ b-1}{c}{x},
\end{align}
where
$$
e_1=2b(c-a)(b-a-1),\ e_3=2a(b-c)(b-a+1),
$$
and
$$
e_2=(b-a)[(1-2x)(b-a-1)(b-a+1)+(b+a-1)(2c-b-a-1)].
$$
With $x=1/2,\ a=-i+j,\ b=i+j+1$,\ $c=2j+2$ and the definition of $l_{i,j}$
we find
\begin{equation*}\label{reone}
 \frac{(i+j+2)(i+1-j)i}{\sqrt{(2i+3)(2i+1)}j(j+1)} (C^n_1)_{i+1,j}=(C^n_1)_{i,j}-\frac{(i+j)(i-j-1)(i+1)}{\sqrt{(2i-1)(2i+1)}j(j+1)} (C^n_1)_{i-1,j},
\end{equation*}
where the top line of equation~\eqref{rep1} has been used. 
Since $(i+j+2)(i-j-1)=(i+1)(i+2)-j(j+1)$ we see that the above equation can
be recast as a generalized eigenvalue equation,
\begin{equation}\label{geneigi}
A_i(C^n_1)_{i,j}=j(j+1) B_i (C^n_1)_{i,j}, 0\le j\le i<n,
\end{equation}
where
\begin{equation}\label{eqia}
A_i=\frac{i(i+1)(i+2)}{\sqrt{(2i+3)(2i+1)}}E_+
+\frac{(i-1)(i)(i+1)}{\sqrt{(2i+1)(2i-1)}}E_- ,
\end{equation}
and
\begin{equation}\label{eqib}
B_i=\frac{i}{\sqrt{(2i+3)(2i+1)}}E_+ + 1+\frac{i+1}{\sqrt{(2i+1)(2i-1)}}E_-.
\end{equation}
Here $E_{\pm}$ are the forward, backward shifts in $i$ respectively. In the
above equation the fact that $(C^n_1)_{i,j}=0$ for $i<j$ has been used.
To obtain a recurrence for fixed $i$ substitute $x=2,\ a=-i+j,\ b=-i-j-1$
and $c=-2i$ which when coupled with the second line in
equation~\eqref{rep1} yields,
\begin{align*}
&\frac{(i+j)(j+1)(i-j+1)}{\sqrt{(2j+1)(2j-1)}}(C^n_1)_{i,j-1}-3(j+1)j(C^n_1)_{i,j}\nonumber\\&+
\frac{(i+j+2)j(i-j-1)}{\sqrt{(2j+1)(2j+3)}}(C^n_1)_{i,j+1}=-i(i+1)(C^n_1)_{i,j}.
\end{align*}
This also can be recast as the generalized eigenvalue equation,
\begin{equation}\label{geneigj}
\hat A_j(C^n_1)_{i,j}=i(i+1) \hat B_j (C^n_1)_{i,j}, 0<j\le i<n
\end{equation}
where
\begin{equation}\label{eqiah}
\hat A_j=\frac{j(j+1)(j+2)}{\sqrt{(2j+3)(2j+1)}}\hat E_+
+3j(j+1)1+\frac{(j-1)(j)(j+1)}{\sqrt{(2j+1)(2j-1)}}\hat E_- ,
\end{equation}
and
\begin{equation}\label{eqibh}
\hat B_j=\frac{j}{\sqrt{(2j+3)(2j+1)}}\hat E_+ +
1+\frac{j+1}{\sqrt{(2j+1)(2j-1)}}\hat E_-.
\end{equation}
Here $\hat E_{\pm}$ are the forward, backward shifts in $j$
respectively. As above we use the condition that $(C^n_1)_{i,j}=0$ for $i<j$.
An interesting formula may be found by eliminating  $p_{j}(2t-1)$  in
\eqref{I1} using equation~\eqref{lrec} which gives, 
$$
I^1_{i,j}=-I^1_{i,j-1}-\frac{(j-1)^2}{(2j-1)(2j-3)}I^1_{i,j-2}+2\int_0^1 tp_i(t)p_{j-1}(2t-1)dt.
$$
Now eliminating $tp_i(t)$ yields,
\begin{equation*}
 I^1_{i,j}=-I^1_{i,j-1}-\frac{(j-1)^2}{(2j-1)(2j-3)}I^1_{i,j-2}+ 2I^1_{i+1,j-1}+\frac{2i^2}{(2i+1)(2i-1)}I^1_{i-1,j-1}.
\end{equation*}
The first line of equation~\eqref{c1ij} yields after increasing
$j$ by one,
\begin{equation}\label{recurij}
\tilde A_j(C_1^n)_{i,j}=\tilde B_i(C_1^n)_{i,j},\quad 0\le j\le i<n,
\end{equation}
where
\begin{equation}\label{eqat}
\tilde A_j=\frac{j}{\sqrt{(2j+1)(2j-1)}}\hat E_- +1 +\frac{j+1}{\sqrt{(2j+3)(2j+1)}}\hat E_+,
\end{equation}
and
\begin{equation}\label{eqbt}
\tilde B_i=\frac{2i}{\sqrt{(2i+1)(2i-1)}}E_- +\frac{2(i+1)}{\sqrt{(2i+3)(2i+1)}}E_+.
\end{equation}

With the above we formulate,
\begin{theorem}\label{recurrence} Let $C^n_1$ and $C^n_{-1}$ be as in
  Theorem~\eqref{repres}. Then they satisfy the generalized eigenvalue
  problems given in equations~\eqref{geneigi} and \eqref{geneigj} and the
  difference equation \eqref{recurij}.

\end{theorem}

\section{The Fourier Transform}

An important object in wavelet theory is the Fourier transform of the
scaling functions. To exploit the symmetry of the Legendre polynomials we
will use equation~\eqref{refsym} and define
$$
\tilde P_n(a)=\int_{-\infty}^{\infty}
e^{-iat}\Phi_n(\frac{t+1}{2})dt=\int_{-1}^{1} e^{-iat} P_n(t)dt,
$$
so that,
\begin{equation}\label{ft}
\tilde P_n(a)=T_n(a)\tilde P_n(a/2),
\end{equation}
where
\begin{equation}\label{pa}
T_n(a)=(C^n_{-1} e^{ia/2}+C^n_{1} e^{-ia/2})/2.
\end{equation}
Since (see \cite{gr}),
\begin{equation}\label{jn}
\int_{-1}^1 e^{iat}\hat p_n(t)=\sqrt{2n+1}\sqrt{2\pi} i^n J_{n+1/2}(a/2)/\sqrt{a},
\end{equation}
where $J_{\nu}$ is the Bessel function of order $\nu$, we obtain the
addition formula
\begin{align}\label{addit}
&\sqrt{2j+1} i^j \frac{J_{j+1/2}(a)}{\sqrt{a}}\nonumber\\&=\frac{1}{2}\sum^j_{k=0}(C^n_1)_{j,k}((-1)^{j+k} e^{ia/2}+e^{-ia/2}) i^k\sqrt{2k+1} \frac{J_{k+1/2}(a/2)}{\sqrt{\frac{a}{2}}},
\end{align}
where the symmetry properties of entries of  $C^n_{-1}$ have been used.
Thus for $j$ even, 
\begin{align*}
\sqrt{4j+1} (-1)^j
\frac{J_{2j+1/2}(a)}{\sqrt{a}}&=\cos(a/2)\sum^j_{k=0}(-1)^k(C^n_1)_{2j,2k}
\sqrt{4k+1} \frac{J_{2k+1/2}(a/2)}{\sqrt{\frac{a}{2}}}\\&+\sin(a/2)\sum^{j-1}_{k=0}(-1)^k(C^n_1)_{2j,2k+1}
\sqrt{4k+3} \frac{J_{2k+3/2}(a/2)}{\sqrt{\frac{a}{2}}},
\end{align*}
while for $j$ odd,
\begin{align*}
\sqrt{4j+3} (-1)^j
\frac{J_{2j+3/2}(a)}{\sqrt{a}}&=-\sin(a/2)\sum^j_{k=0}(-1)^k(C^n_1)_{2j+1,2k}
\sqrt{4k+1}\frac{J_{2k+1/2}(a/2)}{\sqrt{\frac{a}{2}}}\\&+\cos(a/2)\sum^{j-1}_{k=0}(-1)^k(C^n_1)_{2j+1,2k+1}
\sqrt{4k+3} \frac{J_{2k+3/2}(a/2)}{\sqrt{\frac{a}{2}}},
\end{align*}
Recurrence formulas may also be obtained using the fact that Bessel functions
satisfy a differential difference equation. Multiply equation~\eqref{ft} by $\sqrt{a}$
for  $a>0$ and set 
\begin{equation}\label{hft}
\hat P_n(a)=\left[\begin{matrix} J_{1/2}(a)&\cdots&
      i^n\sqrt{2n+1}J_{n+1/2}(a)\end{matrix}\right]^T=G_n\J_n(a),
\end{equation}
where
\begin{equation}\label{gn}
 G_n=\text{diagonal}(i,\ldots,i^n\sqrt{2n+1}),
\end{equation}
and
\begin{equation}\label{cjn}
\J_n(a)=\left[\begin{matrix} J_{1/2}(a)&\cdots& J_{n+1/2}(a)\end{matrix}\right]^T.
\end{equation}
With the above substitions equation~\eqref{ft} becomes,
\begin{equation}\label{hfta}
\hat P_n(a)=\sqrt{2}T_n(a)\hat P_n(a/2).
\end{equation}
Differentiation of $\hat P_n$ and the use of the differential difference
relation $2J_{n+1/2}(a)'=J_{n-1/2}(a)-J_{n+3/2}(a)$ yields, 
\begin{align*}
2\hat
P_n(a)'&=2G_n\J_n(a)'=G_nL_n\J_n(a)+G_n[J_{-1/2}(a),0\ldots,0,-J_{n+3/2}(a)]^T\\&=G_nL_nG_n^{-1}\sqrt{2}T_n(a)\hat
P_n(a/2)+G_n[J_{-1/2}(a),0\ldots,0,-J_{n+3/2}(a)]^T,
\end{align*}
where $L_n$ is an $(n+1)\times(n+1)$ tridiagonal matrix which is  $-1$ on the
upper diagonal $0$ on the diagonal and $1$ on the lower diagonal and
equations~\eqref{hfta} and \eqref{hft} have been used to obtain the last
equality. Differentiation of the right hand side of equation~\eqref{hfta} then using similar manipulations as above yields
\begin{align}\label{dfta}
T_n(a)'\hat P_n(a/2)=&\left(H_nT_n(a)-\frac{1}{2}T_n(a)H_n\right)\hat P_n(a/2)\\\nonumber&+
\frac{1}{\sqrt{2}}G_n[J_{-1/2}(a),0,\ldots,0,-J_{n+3/2}(a)]^T\\\nonumber& -\frac{1}{2}P_n(a)G_n[J_{-1/2}(a/2),0\ldots,0,-J_{n+3/2}(a/2)]^T,
\end{align}
where $H_n=G_nL_nG_n^{-1}$. Examination of the above equation for $a$ small
and positive shows that for fixed $j$ the sequence $((-1)^{j+k} e^{-ia/2}+e^{ia/2})J_k(a/2),k=0,\ldots,n$ is linearly independent. Thus the above equation yields the difference equation,
\begin{align}\label{pdde}
K_i(C^n_1)_{i,j}=J_j(C^n_1)_{i,j}, \ 0<i<j<n,
\end{align}
where
$$
K_i=\sqrt{\frac{2i+1}{2i-1}}E_-+ \sqrt{\frac{2i+1}{2i+3}}E_+ ,
$$
and
$$
J_j=\frac{1}{2}\sqrt{\frac{2j-1}{2j+1}}\hat E_-+ \frac{1}{2}\sqrt{\frac{2j+3}{2j+1}}\hat E_+ + 1.
$$

\section{Wavelets}
We now develop equations to compute a set of orthogonal wavelets associated with the above scaling functions. We are interested in finding wavelet functions that form a basis for $L^2(R)$
and are obtained by integer translates and dilations by 2 of a fixed set
of functions. From equation~\eqref{nest} with the change of variable that
lead to \eqref{refsym} then for approximation order $n$ it is enough to find $(n+1)\times(n+1)$ matrices $D_{-1}$ and $D_1$, and functions  
$$\Psi_n=\left(\begin{matrix}\psi_0^n &\cdots&\psi^n_n\end{matrix}\right)^T
$$ 
given by
\begin{align}\label{wrefine}
 \Psi_n(\frac{t+1}{2})&=D^n_{-1}\Phi_n(t+1)+D^n_1\Phi_n(t)\nonumber\\&=D^n_{-1}P_n(2t+1)+D^n_1P_n(2t-1),
\end{align}
where the last equality holds for $-1\le t\le 1$. The imposed orthogonality implies,
\begin{equation}\label{cdortho}
 C^n_{-1} {D^n_{-1}}^T + C^n_1 {D^n_{1}}^T=0,
\end{equation}
and
\begin{equation}\label{ddortho}
 D^n_{-1} {D^n_{-1}}^T + D^n_1{D^n_{1}}^T=2I_{n+1}.
\end{equation}
From \eqref{wrefine} we find
$$
D^n_1=\int_0^1\Psi_n(t) P_n(2t-1) dt,
$$
and
$$
D^n_{-1}=\int_{-1}^0\Psi_n(t)P_n(2t+1) dt.
$$
 For general $n$ there are an infinite number of solutions to the above equations even if we ask that the wavelet functions in $\Psi_n$ 
be symmetric or antisymmetric. If we solve equations~\eqref{cdortho} and \eqref{ddortho} with $n=0$ we find that
 $(D^0_{-1})_{0,0}=-(D^0_{1})_{0,0}=(C^0_1)_{0,0}$ so that the first wavelet function is the Haar wavelet which is 
antisymmetric. Thus to obtain symmetry set $(D^n_{-1})_{i,j}=(-1)^{i+j+1}(D^n_{1})_{i,j},\ 0\le i, \ j\le n$. For $n=1$ we find
$$
D^1_1=\left(\begin{matrix} (D^1_1)_{0,0} &  (D^1_1)_{0,1}\\ (D^1_1)_{1,0} &  (D^1_1)_{1,1}\end{matrix}\right)
$$
and
$$
D_{-1}^1=\left(\begin{matrix} (-D^1_1)_{0,0} &  (D^1_1)_{0,1}\\ (D^1_1)_{1,0} &  (-D^1_1)_{1,1}\end{matrix}\right)
$$
If we insist that $D^1_1$ has positive diagonal entries there is a unique solution to equations~\eqref{cdortho} and \eqref{ddortho} given by
$$
D^1_1=\left(\begin{matrix} (C^1_1)_{1,1} &  -(C^1_1)_{1,0}\\ 0 &  1\end{matrix}\right).
$$
This suggests that a unique solution can be found for which $D^n_1$ is upper triangular with positive diagonal entries.
\begin{theorem}\label{thmwave}
 Let $C_1$ be a lower triangular matrix with positive diagonal entries  satisfying $C_{-1} C_{-1}^T+C_1 C_1^T=2I$ where $C_{-1}$ be obtained from $C_1$ by the
symmetry relation $(C_{-1})_{i,j}=(-1)^{i+j}(C_1)_{i,j}$. Then for $n\ge1$
there is a unique upper triangular $(n+1)\times(n+1)$ matrix $D_1$ with positive diagonal entries that satisfies equations~\eqref{cdortho} and \eqref{ddortho} where $D_{-1}$ has the symmetry relations
$(D_{-1})_{i,j}=(-1)^{i+j+1}(D_1)_{i,j}$ and $(D_{1})_{n,n}=1$
\end{theorem}
\begin{proof}
 We note that the result is true for $n=1$ so we suppose it is true by induction for $n-1$. Consider the $n\times n$ matrices $\hat C_1$ obtained from $C_1$ by deleting the first row and column. Then from the induction hypothesis there is a unique upper triangular $\hat D_1$ associated with $\hat C_1$ which satisfies equations~\eqref{cdortho} and \eqref{ddortho} and  $(\hat D_{1})_{n-1,n-1}=1$. Let  ${\bf c_0}$ be the first column of $C_1$, $\hat{\bf c_i}, i=1,\ldots, n$ be the rows of $\hat C_1$, and write
$$
C_1=\left(\begin{matrix} {\bf c_0} & 0\\ &\hat C_1\end{matrix}\right).
$$
Likewise let ${\bf d_0}$ be the first row of $D_1$  and  
$\hat{\bf d_i}, i=1,\ldots, n$ be the  rows of $\hat D_1$.
Using the symmetry equations we see that \eqref{cdortho} and \eqref{ddortho} 
yield the equations,
\begin{equation}\label{cceq}
(({\bf c_0})_i, \hat{\bf c_{i}}){\bf d_0}^T=0,\ i=1,3,\ldots,
\end{equation}
\begin{equation}\label{ddeq}
(0, \hat {\bf d_i}){\bf d_0}^T=0,\ i=2,4,\ldots.
\end{equation}
and
\begin{equation}\label{normal}
{\bf d_0}{\bf d_0}^T=1.
\end{equation}
Since $\hat C_1$ is lower triangular with positive diagonal elements the vectors $({\bf c_0}_i, \hat{\bf c_{i}})$ in equation~\eqref{cceq} are independent. The equations  $\hat C_{-1}(\hat D_{-1})^T+\hat C_1(\hat D_1)^T=0$, and 
$\hat D_{-1}(\hat D_{-1})^T+\hat D_1(\hat D_1)^T=0$ show that the vectors
$(0, \hat {\bf d_i})$ in \eqref{ddeq} are orthogonal to each other and to
the vectors in \eqref{cceq}. Thus the rank of the matrix whose rows are the 
equations~\eqref{cceq} and \eqref{ddeq} is $n$. With $i=1$ in \eqref{cceq} we find $(C_1)_{0,1} (D_1)_{0,0}+ (C_1)_{1,1} (D_1)_{0,1}=0$ which implies that $(D_1)_{0,0}$ is the free variable which is fixed uniquely by the choosing the positive solution to
 equation~\eqref{normal} .
\end{proof}
We now show,
\begin{lemma}\label{even}
Given $C^n_i,\ i=\{-1,1\}$ suppose $D^n_1$ and $D^1_1$ satisfy the hypothesis of the
Theorem~\eqref{thmwave}. Then $(D^n_1)_{n-2j,n}=0$ and $(D^n_1)_{n-2j,n-k}=(D^{n-1}_1)_{n-2j,n-k}$.
\end{lemma}
\begin{proof} We begin with the observation that from the symmetry
  relations we find $(D^n_{-1})_{n-2j,n}= -(D^n_1)_{n-2j,n}$. Thus last row
  of equation~\eqref{ddortho} shows that
  $(D^n_{-1})_{n-2j,n}=0$. The proof of Theorem~\eqref{thmwave} also shows
  that in order to compute $(D^n_1)_{n-2k,i} , i=n-2k,..,n-1$, we can choose
  $\hat D_1$ and $\hat D_{-1}$ so that they start with  the row $n-2j+1$
  (starting from zero) of $D^n_1$ and $D^n_{-1}$. Examination of the equations~\eqref{cceq} and \eqref{ddeq} yield,
$$
(C^n_1)_{n-i,n-i-1} (D^n_1)_{n-i-1,n-i-1}+(C^n_1)_{n-i,n-i}(D^n_1)_{n-i-1,n-i}=0,
$$
and
$$
(D^n_1)^2_{n-i-1,n-i-1}+(D^n_1)^2_{n-i-1,n-i}=1,
$$
for $i=0,1$. The unique solutions of these equations from
Theorem~\eqref{thmwave} and the entries of $(C^n_1)$ in
equation~\eqref{cnm1} above are given respectively by
equations~\eqref{dnm1} and \eqref{dnm2} below and shows explicitly that the
result for $(D^n_1)_{n-2,n-k}, k=1,2$. Using that
$(C^n_1)_{n-j,n-k}=(C^{n-1}_1)_{n-j+1,n-k+1}$ for $j>0,k>0$ in equations~\eqref{cceq} and the induction hypothesis in equations~\eqref{ddeq} imply that the entries $(D^n_1)_{n-2j,n-k}$ solve the same equations as  $(D^{n-1}_1)_{n-2j,n-k}$ for $k=1,\ldots, 2j$. The uniqueness of the solutions given by Theorem~\eqref{thmwave} above proves the Lemma.
\end{proof}

Using Theorem~\eqref{thmwave} allows us to compute some of the matrix
elements in $D^n_1$. To this end we find 
for row $n+1$,
$$(D_1^n)=1,$$
for  row $n$,
\begin{equation}\label{dnm1}
(D^n_1)_{n-1,n-1}=\frac{1}{2n},\  (D^n_1)_{n-1,n}=-\frac{\sqrt{(2n+1)(2n-1)}}{2n},
\end{equation}
for row $n-1$ 
\begin{equation}\label{dnm2}
(D^n_1)_{n-2,n-2}=\frac{1}{2n-2},\ (D^n_1)_{n-2,n-1}=-\frac{\sqrt{(2n-1)(2n-3)}}{2n-2},\ (D^n_1)_{n-2,n}=0,
\end{equation}
for row $n-2$,
\begin{align*}
&(D^n_1)_{n-3,n-3}=\frac{3}{4(n-1)(n-2)},\ (D^n_1)_{n-3,n-2}=-\frac{3\sqrt{(2n-3)(2n-5)}}{4(n-1)(n-2)},\\& 
(D^n_1)_{n-3,n-1}=\frac{(2n+1)\sqrt{(2n-1)(2n-5)}}{4(n)(n-1)},\ (D^n_1)_{n-3,n}=\frac{\sqrt{(2n+1)(2n-5)}}{4(n)(n-1)}.
\end{align*}
\section{Acknowledgements}
JSG would like to thank Eric Koelink for suggesting the transformation
leading to equation~\eqref{rep2} and Plamen Iliev for discussions on the
generalized eigenvalue problem. JSG would like to thank the University
Carlos III de
Madrid  and especially the Mathematics Department where he was a C\'atedra de Excelencia for its support and hospitality.


\end{document}